\documentclass[reqno]{amsart}

\begin{document}
\title{Painlev\'e's determinateness theorem extended to proper coverings}

\author{Claudi Meneghin} 

\address{Claudi Meneghin (Isrc) - 
Fermo Posta Chiasso 1, \newline CH-6830 Chiasso - Switzerland}
\email{claudi.meneghin@gmail.com}


\font\sdopp=msbm10 
\font\sdoppet=msbm8
\font\cir=wncyb10
\def\ER {\sdopp {\hbox{R}}}
\def\CI {\sdopp {\hbox{C}}}
\def\Cit {\sdoppet {\hbox{C}}}

\begin{abstract}
We extend Painlev\'e's determinateness theorem
to the case of first order ordinary differential equations in the complex domain
with known terms allowed be multivalued in the dependent variable as well;
multivaluedness is supposed to be resolved by proper coverings.
\end{abstract}

\maketitle
\numberwithin{equation}{section}
\newtheorem{theorem}{Theorem}[section]
\allowdisplaybreaks

\bibliographystyle{plain}

\def\QUAN{\vrule height6pt width6pt depth0pt}
\newtheorem{definition}{Definition}
\newtheorem{lemma}[definition]{Lemma}

\font\sdopp=msbm10 
\font\sdoppet=msbm8
\font\cir=wncyb10
\def\ER {\sdopp {\hbox{R}}}
\def\CI {\sdopp {\hbox{C}}}
\def\Cit {\sdoppet {\hbox{C}}}
\def\DI {\sdopp {\hbox{D}}}
\def\EN{\sdopp {\hbox{N}}}
\def\ZETA{\sdopp {\hbox{Z}}}
\def\pen{\sdoppet {\hbox{N}}}
\def\PI {\sdopp {\hbox{P}}}
\def\R{\Delta} 
\def\NORM{\hbox{\boldmath{}$\vert\vert$\unboldmath}}
\font\cir=wncyb10
\def\Iu{\mathcal{U}}
\def\Ze{\cir\hbox{Z}}
\def\pe{\cir\hbox{P}}
\def\Ef{\cir\hbox{F}}
\def\M{\hbox{\tt\large M}}

\section{Foreword}
What is generally referred to (see e.g.$\!$ \cite{hille}, th.3.3.2) as {\bf Painlev\'e's Determinateness Theorem} for 
first order ordinary differential equations in the complex domain states: 
\vskip0.1cm
{\it
If $F(z,w)$ is a rational function of $w$ with coefficients which are algebraic functions of $z$, then any movable singularities of the solutions to the first order ODE $w' = F(z,w)$ are poles and/or algebraic branch points.
}
\vskip0.1cm

The so called 'known term' $F$ is required to be 'single-valued' in the 'dependent variable' $w$; the goal of this note is to allow $F$ 
to be 'multivalued' in $w$ as well and conclude that none of the movable singularities are essential notwithstanding. 
The multivaluedness of $F$ will be resolved by passing to a  Riemann domain $(\Delta,p)$ over $\CI^2$ minus a complex-analytic curve,
with the main assumption that $(\Delta, p)$ be {\it proper cover} of $p(\Delta)$.
This hypothesis cannot be dropped in general, as example (\ref{noproper}) shows.
More singularities (not necessarily poles) for $F$ will be allowed on a complex-analytic curve in $\Delta$.

Previous statements of Painlev\'e's theorem (\cite{painleve'}, p.38, \cite{picard}, p.327/328, \cite{ince}, p.292, \cite{hille}, th.3.3.1) 
read as follows:
\vskip0.1cm
{\it
If a solution of the first order ODE $w' = F(z,w)$ is continued analytically along a rectifiable arc from $z=z_0$ to $z=z_1$ 
avoiding the set $S$ of fixed singularities, and if $z_1\not\in S$, then the solution tends to a definite limit, finite or infinite, as $z\to z_1$.
}
\vskip0.1cm

The two statements of Painlev\'e's determinatess theorem are equivalent under the assumption about the known term $F$ (see the references here above); 
since we make a broader hypothesis, we can no longer take equivalence for granted; in particular, logarithmic branch points in the solutions cannot be 
excluded, as pointed out in sect.$\!$ 3.1 of \cite{hille} (see also example \ref{Solution with logarithmic singularity} in this article). 
In this note we generalise this latter version, but the final remark shows that movable essential singularities can be ruled out anyway;
the question whether or not, in our broader setting, natural boundaries could arise will be the object of future investigations.

This note does not require Bieberbach's (\cite{bieberbach}, quoted in \cite{steinmetz}) precise definition of a fixed or of a movable singularity;
we will use these notions in an informal fashion (like in \cite{hille,ince, painleve',picard}), within the examples only.

Before stating and proving our main theorem,
we now introduce some terminology and discuss some examples.

\section{Terminology}

A {\it Riemann domain} over a region 
$\mathcal{U}\subset\CI^n$ is a complex manifold $\R$ with
an everywhere maximum-rank holomorphic surjective mapping 
$p:\R\rightarrow\mathcal{U}$; $\R$ is {\it proper} provided that so is $p$
(see \cite{gunros} p.43); a curve $\mathcal{S}\subset\CI^n$ is {\it complex-analytic}
provided that it is the common zero set of $N-1$ complex-analytic functions on $\CI^n$; 
when $n=1$ we talk about {\it Riemann surfaces}. 

\begin{definition}
Let $\M$ be a complex manifold, $U$ an open set in $\CI^n$, $f:U\to\M$ an holomorphic mapping: a 
{\rm regular analytic continuation} of the holomorphic mapping element $\left(U,f\right)$, is a quadruple 
$\left(S,\pi,j,F\right)$ such that:
{\bf 1)} $S$ is a connected Riemann domain over a region 
in $\CI^n$; 
{\bf 2)} $\pi\,\colon\, S\rightarrow \CI^n$ is an 
everywhere nondegenerate holomorphic mapping such 
that $U\subset \pi(S)$;
{\bf 3)}
$j\,\colon\, U\rightarrow S$ is a holomorphic immersion such 
that $\pi\circ j=id\vert_{U}$;
{\bf 4)}
$F\,\colon\, S\rightarrow \M$ is a holomorphic mapping such 
that $F\circ j=f$.
\label{continuation}
Let $\gamma:I\to\CI^n$ be an arc 
{\rm (}with $I=[0,1]$ or $I=[0,1)${\rm )}
such that $\gamma (0)=X$; a {\rm regular analytic continuation along $\gamma$} 
of $\left(U,f\right)$ is a regular analytic continuation $\left(S,\pi,j,F\right)$
of $\left(U,f\right)$ such that there exists an arc
$\widetilde \gamma:I\to S $ with 
$\pi \circ \widetilde \gamma=\gamma $.
\label{arc continuation}
\end{definition}

\subsection{Cauchy's problems with multivalued known terms}
\label{multivalued known terms}
Let us now focus on differential equations whose
'known terms' are defined on Riemann domains rather than just on open sets
in $\CI^2$. Introduce the following:
\begin{itemize}
\item a complex-analytic curve
$\mathcal{B}=\{(z,w)\in\CI^2\colon B(z,w)=0\}$, 
with $ B$ holomorphic on $\CI^2$, called the {\bf branch
locus} of the differential equation, and a proper Riemann domain
$(\Delta , p)$ over $\CI^2\setminus\mathcal{B}$. Note that we do not 
require $\mathcal{B}$ to be algebraic. 
\item
a complex one-dimensional submanifold $\Sigma \subset \Delta$,
such that 
$p(\Sigma)$ 
is included 
in a complex-analytic curve $\Lambda (z,w)=0$ in $\CI^2$;
 $p(\Sigma )$ will be referred to
as the {\bf singularity locus}; 
\item
the branch and the singularity loci will be collectively 
referred to as {\bf the singularities of the differential
equation};
\item
a holomorphic function $F$ on $\Delta \setminus \Sigma$,
called the {\bf known term};
\item
a point $X_0\in \Delta \setminus{\Sigma }$, 
with $(z_0,w_0):=p(X_0)$; we will refer to $w_0$ as the {\bf initial value}
of the Cauchy problem and to $z_0$ as the {\bf initial point}; we will also  
refer to $(z_0,w_0)$ collectively as the {\bf initial values};
\item
a local inverse $\eta $ of $p$, defined  in a 
bidisc $\DI_1\times\DI_2 $ around $(z_0,w_0)$.
\end{itemize}

The above definitions are meant to be referred to a differential equation
(or to an associated Cauchy problem) and not to its solutions.

\section{Examples}
In the realm of practice, the usual symbols of 'multivalued
functions' such as '$\log$' or '$\sqrt{\ }$' will
go on to be used as well as the attributes '{\it multi-valued}' or '{\it single-valued}'.
This is perfectly rigorous (even by a geometric point of view), inasmuch as the underlying 
machinery of analytic continuation is understood; in particular, a branch of the multivalued known term
will always have to be specified alongside the initial conditions, i.e., a local inverse $\eta$ of the covering map 
$p:\Delta\to\CI^2$ 
will have to be explicitely chosen there.

\subsection{Attaining singularities of the known term}
\label{leading}

$\!$Consider the following Cauchy problem: 
$$
\begin{cases}
\displaystyle
w'(z)={\sqrt{(1-w^2(z))}\ \frac{w(z)}{\sin z}} \cr
w(\pi/4 )=\sqrt{2}/2.
\end{cases}
$$ 

Here we understand the choice of the positive branch of the square root
corresponding to the initial values 
$(\pi/4, \sqrt{2}/2)$.

In the terminology of section \ref{multivalued known terms}, we have that:
\begin{itemize}
\item 
the branch locus is $\mathcal{B}=\{(z,w)\in\CI^2: {w=\pm 1\}}$, 
the Riemann domain of the known term is 
$\Delta=\{(z,w,y)\in\CI^3:w^2+y^2=1,\ y\not=0\} $, with the projection mapping
$p(z,w,y)=(z,w)$, a twofold covering, hence a proper mapping;
\item
the singularity locus is $\Sigma = \{(z,w,y)\in\Delta:z=k\pi, k\in\ZETA\}$;
\item
the known term $F:\Delta \setminus \Sigma\to\CI$ is defined by 
$F(z,w,y)=yw/\sin(z)$;
\item
the lifted initial point is $X_0=(\pi/4, \sqrt{2}/2, \sqrt{2}/2)\in\Delta\setminus\Sigma$; 
note that $p(X_0)=(z_0,w_0)=(\pi/4, \sqrt{2}/2)$;
\item
$\eta (z,w)= (z,w,\sqrt{1-w^2})$, where the positive branch of the square root has been
chosen.
\end{itemize}
The singularities of the equation in the underlying $\CI^2$ lie on 
$\{z=k\pi\}\cup\{w=\pm 1\}$.
The problem is solved by the entire function
$w(z)=\sin(z)$ (clearly admitting analytic continuation and, a fortiori, limit, everywhere in $\CI$).

Note that the multivaluedness in $w$ of the known term  of this Cauchy problem
makes it attain singularities along the graph of the solution,
more precisely at $z=\frac{\pi }{2}+k\pi$.
This fact does not affect the analytic continuation of the solution since 
the above singularities can be avoided by continuing along a suitable real arc; compare
the argumentation following (\ref{discrete}). 

\subsection{A problem with essential singularities}
The Cauchy problem 
$$
\begin{cases}
\displaystyle
w'(z)=({e^{z\cdot w(z)}+1})^{-1}
\left(
{e^{-z\cdot w(z)}}
-w(z)
\right)
e^{\,({e^{z\cdot w(z)}-z})^{-1}+1}
\cr
w(2)=0
\end{cases}
$$ 
is solved by
$w(z)=[\log(z-1)]/z$. 
The known term of this problem is single valued on
$\CI^2$, has poles on the complex-analytic 
curve ${e^{wz}=-1}$ and essential singularities on 
${e^{wz}=z}$. No line $z={\rm const}$ (in particular $z=1$) is a singularity.
In view of theorem \ref{painleve'}, note that $w$ can be analytically continued
along the arc $\gamma $ defined on $[0,1)$ by $\gamma(t)=2-t$ and
there does exist $\lim_{t\to 1}[\log(1-t)]/(2-t)=\infty$.

\subsection{No limit}
\label{leading1}

Consider the following Cauchy problem: 
$$
\begin{cases}
\displaystyle
w'(z)=-{\sqrt{1-w^2(z)}}/{z^2} \cr
w((1+i)^{-1})=\sin(1+i+c).
\end{cases}
$$ 

Here we suppose $\vert c\vert$ small enough and understand the choice of the positive branch of the square root
corresponding to the initial values $((1+i)^{-1},\ \sin(1+i+c))$.
In the terminology of section \ref{multivalued known terms}, we have that:
\begin{itemize}
\item 
the branch locus is $\mathcal{B}=\{(z,w)\in\CI^2: {w=\pm 1\}}$, 
the Riemann domain of the known term is 
$\Delta=\{(z,w,y)\in\CI^3:w^2+y^2=1,\ y\not=0\} $, with the projection mapping
$p(z,w,y)=(z,w)$, a twofold covering, hence a proper mapping;
\item
the singularity locus is $\Sigma = \{(z,w,y)\in\Delta:z=0\}$;
\item
the known term $F:\Delta \setminus \Sigma\to\CI$ is defined by 
$F(z,w,y)=-y/z^2$;
\item
the lifted initial point is $X_0=(1/\pi, -\sin(c), -\cos(c))\in\Delta\setminus\Sigma$; 
note that $p(X_0)=(z_0,w_0)=(1/\pi, -\sin(c))$;
\item
$\eta (z,w)= (z,w,-\sqrt{1-w^2})$, where the positive branch of the square root has been
chosen.
\end{itemize}
The singularities of the equation in the underlying $\CI^2$ lie on 
$\{z=0\}\cup\{w=\pm 1\}$.
The problem is solved by $w(z)=\sin(1/z+c)$, showing an essential singularity at $z=0$. 
Note that $w$ can be analytically continued
along the arc $\gamma $ defined on $[0,1)$ by $\gamma(t)= (1-t)(1+i)^{-1}$ and there does not exist
$\lim_{t\to 1}\sin(1/\gamma(t)+c)$; in view of theorem \ref{painleve'}, this should be compared with the fact that $\{z=0\}\subset\CI^2$ is a line of 
poles for the known term.

\subsection{Solution with logarithmic singularity}
\label{Solution with logarithmic singularity}
In the Cauchy problem 
$$
\begin{cases}
\displaystyle
w'(z)=e^{-w(z)}\left({1+\sqrt[3]{e^{w(z)}-z+1}}\right)/2\cr
w(1)=0,
\end{cases}
$$ 
we understand 
the choice of the positive branch of the cube root corresponding to the initial values 
$(1,0)$.
In the terminology of section \ref{multivalued known terms}, we have:
\begin{itemize}
\item 
the branch locus is $\mathcal{B}=\{(z,w)\in\CI^2: e^w -z+1=0\}$, 
the Riemann domain of the known term is 
$\Delta=\{(z,w,y)\in\CI^3:e^w-z+1=y^3, \ y\not=0\} $, with the projection mapping
$p(z,w,y)=(z,w)$, a threefold covering, hence a proper mapping;
\item
the singularity locus $\Sigma $ is empty, indeed 
the known term $F:\Delta\to\CI$, defined by 
$F(z,w,y)=e^{-w}(1+y)/2$ is holomorphic on the whole of $\Delta$;
\item
the lifted initial point is $X_0=(1,0,1)\in\Delta$; 
note that $p(X_0)=(z_0,w_0)=(1,0)$;
\item
$\eta (z,w)= (z,w,\sqrt[3]{e^{w(z)}-z+1})$, where the positive branch of the cube root has been
chosen.
\end{itemize}
The singularities of the equation in the underlying $\CI^2$ lie on the curve
$e^{w}-z=-1$.
The problem is solved by 
$w(z)=\log z$, which can be analytically continued
along the arc $\gamma $ defined on $[0,1)$ by 
$\gamma(t)=1-t$. In view of theorem \ref{painleve'}, note that the complex line ${z=0}$
is not a singularity for the differential equation and 
there does exist $\lim_{t\to 1}\log(1-t)=\infty$.

\subsection{A separable 'multivalued' problem}

Consider the following Cauchy problem: 
$$
\begin{cases}
\displaystyle
w'(z)={\sqrt{z}}/{\sqrt{w(z)}} \cr
w(1)=(1+c)^{2/3}.
\end{cases}
$$ 

We have supposed $\vert c\vert$ positive, real and small enough; we have 
chosen the positive branches of the square and cube roots
corresponding to the initial values 
$(1,(1+c)^{2/3})$.

As in the preceeding examples, the Riemann domain
of the known term ${\sqrt{z}}/{\sqrt{w}}$  is proper; 
the underlying singularities of the equation are on $\{z=0\}\cup\{w=0\}$.

The problem is solved by $w(z)=(z^{3/2}+c)^{2/3}$, 
showing a 'fixed' algebraic branch point at $z=0$ and a 'movable' one
at $z^{3/2}+c=0$. 
Note that $w$ can be analytically continued
along the arc $\gamma $ defined on $[0,1)$ by 
$\gamma(t)=1-t$ (without stumbling on any of the movable branch points $z^{3/2}+c=0$ on this path,
since $(1-t)^{3/2}>0$ and $c>0$) and 
$\lim_{t\to 1}[(1-t)^{3/2}+c)]^{2/3}= c^{2/3}$.

In a different fashion,
$w$ admits the following analytic continuation along a path pointing towards one of the 'movable' branch points:
consider the arc, defined on $[0,1]$:    
$$
\beta(t):=
\begin{cases}
\displaystyle
e^{4\pi it} & \text{if $0\leq t\leq 1/2 $}\cr
2(1-t)+c^{2/3}(2t-1) & \text{if $1/2\leq t\leq 1 $}.
\end{cases}
$$ 
After that $t$ has run on $[0,1/2]$, the analytic continuation of 
$w(z)=(z^{3/2}+c)^{2/3}$ has changed sign from positive to negative; further running the parameter on
$[1/2,1]$ makes $z$ run into $c^{2/3}$, which is this time a branch point for 
$(z^{3/2}+c)^{2/3}$ on this path,
since $[2(1-t)+c^{2/3}(2t-1)]^{3/2}<0$ and $c>0$.
All the same, notice that $\lim_{t\to 1}\{[2(1-t)+c^{2/3}(2t-1)]^{3/2}+c\}^{2/3}=0$
on this branch i.e., the analytic continuation of the solution of our Cauchy problem admits limit even in the above circumstance.

\subsection{Counterexample: a logarithmic known term}

Consider the following (autonomous) Cauchy problem: 
$$
\begin{cases}

\displaystyle
w'(z)=-w(z) \log^2(w(z)) \cr
w(0)=e^{-1/c}.
\end{cases}
\label{noproper}
$$ 

This problem is solved by $w(z)=e^{1/(z-c)}$. Notice that
the multivaluedness in $w$ of the known term is logarithmic.
In the terminology of section \ref{multivalued known terms}, we have:
\begin{itemize}
\item 
the branch locus is $\mathcal{B}=\{(z,w)\in\CI^2: w=0\}$, 
the Riemann domain of the known term is 
$\Delta=\{(z,w,y)\in\CI^3:w=e^y\} $, with the projection mapping
$p(z,w,y)=(z,w)$, which is not a proper covering;
\item
the singularity locus $\Sigma $ is empty, indeed 
the known term $F:\Delta\to\CI$, defined by 
$F(z,w,y)={-w}y^2$ is holomorphic on the whole of $\Delta$;
\item
the lifted initial point is $X_0=(0, e^{-1/c}, -1/c)\in\Delta$; 
note that $p(X_0)=(z_0,w_0)=(0, e^{-1/c})$;
\item
$\eta (z,w)= (z,w,\log w)$, where the real branch of the logarithm has been
chosen.
\end{itemize}
The singularities of the equation in the underlying $\CI^2$ lie on the curve
${w}=0$.
In view of theorem \ref{painleve'}, note that the known term is not resolved by
a proper cover and that the solution has a movable essential singularity; alternatively, 
it can be stated that there exists a path $\gamma$ defined on $[0,1]$, joining $0$ and $c$ and such that
$w$ admits analytic continuation along
$\gamma\vert_{[0,1)}$ but no continuous extension up to $\gamma(1)=c$.
\vskip0.3cm
\noindent
{\bf Remark} This example shows that, in general, the hyphotesis in theorem \ref{painleve'} that the 
multivaluedness of the known term be resolved by a proper cover cannot be dropped; however,
a first order o.d.e.$\!$ with nonproper covering associated to its known term
can yield a family of functions free from movable singularities notwithstanding.
For instance, the equation 
$$
w'(z)=\sqrt{1-w^2(z)}\ \frac{\arcsin(w(z))}{z}
$$
admits the family of solutions $\{\sin(kz)\}_{k\in\Cit}$, which are free from any singularities at all.

\section{The main theorem}

Now we are ready to introduce the main issue of this paper:
in the terminology of section \ref{multivalued known terms}, introduce the
(well defined) 
Cauchy problem:
\begin{equation}
\begin{cases}
u^{\prime}(v)
=F\circ\eta(z, w(z))
\cr
w(z_0)=w_0.
\end{cases}
\label{cauchy}
\end{equation}

Note that the above Cauchy problem is of a 'classical' type, i.e., the known term is defined 
on an open set $\mathcal{U}\subset\CI^2$ and a solution is sought that be a holomorphic function on a 
one-dimensional complex disc and whose graph is contained in $\mathcal{U}$.

Thanks to the classical existence-and-uniqueness theorem
(see e.g.$\!$ \cite{hille}, th 2.2.2, \cite{ince} p.281-284), such a solution does exist.
The problem of its analytic continuation is natural and settles (besides the usual matters dealing with 
the analytic continuation of a function of one complex variable) a supplementary 
question, i.e.$\!$, what happens if the analytic continuation $\omega$ of the graph of the solution leads to 
singularities in the known term i.e., for instance, points where $F\circ\eta $ is not holomorphic? 
Let $\gamma$ be an arc defined on $[0,1]$: if the Riemann domain $(\Delta, p)$ resolving the multivaluedness of $F$ is proper
and the complex line $\{w=\gamma(1)\}$ is not a singularity, theorem \ref{painleve'} answers that 
the feasibility of analytic continuation along $\gamma$ restricted to the semi-open interval $[0,1)$ 
entails the existence of a (finite or infinite) limit for $\omega\circ\gamma$ as the arc parameter tends to $1$.

Now we need a technical lemma:
\begin{lemma}
Let $X$ be a metric space, $\alpha : [0,1)\to X$ a continouous 
arc such that  $\lim_{t\to 1}\alpha (t)$ does not exist
in $X$.
{\tt A)}let $\{x_l\}\to x_{\infty }$ be an injective converging  
sequence in $X$: then there exists a 
sequence $\{t_k\}\to 1$ and an open neighbourhood $U$
of $\{x_l\}\cup \{x_{\infty }\}$ such that 
$\{\alpha (t_k)\}\subset X\setminus U$. 
{\tt B)}for every 
$N-$tuple $\{x_1...x_N\}\subset X$ there exists a sequence
$\{t_i\}\rightarrow b$ and neighbourhoods
$U_k$ of $x_k$ such that $\{\alpha (t_i)\}\subset X\setminus 
\bigcup_{k=1}^N U_k$.
\label{fmetric}
\label{metric}
\end{lemma}
\begin{proof}
{\tt A)} Since none of the $\{x_l\}$'s ($l\in\EN\cup \{\infty \}$) 
is $\lim_{t\to 1}\gamma(t) $, we have that for every $l\in\EN\cup \{\infty \}$ 
there exists an open neigbhourhooud
$V_l$ of $x_l$ such that 
$
\alpha ([\lambda ,1))\not\subset V_l
$ 
for every $\lambda \in[0,1)$; moreover, up to shrinking $V_{\infty }$, 
there exixts $N>1$ such that $n>N\Rightarrow x_n\in   V_{\infty }$
but $x_N\not\in   V_{\infty }$. Clearly we can choose the 
$\{V_l\}$'s in such a way that: $V_i\cap V_k=\emptyset$
if $i,k\in\EN$ and $i\not=k$; $V_i\cap V_{\infty }=\emptyset $ if $l\leq N$.
Let now $U:=(\bigcup_{l=1}^N V_l)\cup V_{\infty }$: 
by construction, $U$ is disconnected. 
Since $\alpha ([\lambda ,1 )) $ is, by contrast, connected, and, by construction, $\alpha ([\lambda ,1 )) $
is not contained in a single connected component of $U$,
we must have $\alpha ([\lambda ,1))\not\subset U$; this entails that, for every $k>0$, the set 
$W_k:=\alpha^{-1}(X\setminus U)\cap(1-1/k,1)$ is not empty;
picking $t_k \in W_k$ ends the proof. 
The proof of {\tt B)} is analogous and will be omitted.
\end{proof}

We also need to generalise Hille's theorem 3.2.1 \cite{hille} to our broader setting:
we will use once more the notation of section \ref{multivalued known terms}.
Note that in the proof of this lemma we are forced to work first in the underlying environments
$\CI$ and $\CI^2$ and to lift the results by local charts into the overlying Riemann surface $R$ and domain $\Delta$. 
This is why derivation is defined on $\CI_z$ and the analytic continuation of the solution
takes values in $\CI_w$. Also, the Taylor developments (\ref{taylor}) are feasible using local charts in $\CI^2$ 
and not directly in the complex manifold $\Delta$. A similar approach is implicit in Hille's proof.

\begin{lemma}
\label{single sequence} 
Let the Cauchy problem (\ref{cauchy}) be given. 
Suppose that $\gamma\colon [0,1]\rightarrow\CI$
is an arc starting at the initial point $z_0$ and that an analytic continuation 
$\left(R,\pi,j,\omega\right)$ of the initial solution $w$
can be carried out along $\gamma\vert_{[0,1)}$; let $\widetilde\gamma :[0,1)\to R$ be the lifted arc of
$\gamma $ with respect to the natural projection $\pi $.
Consider the arc 
$\theta:=\gamma
\times
[\omega \circ 
\widetilde\gamma]
:[0,1)\to\CI^2$; suppose that the initial known term $F\circ\eta$ can be analytically continued along
$\theta$ and let 
$\widetilde\theta:[0,1)\to\Delta$ be the lifted arc with respect to the natural projection $p:\Delta\to\CI^2$.
Finally, suppose that there exists a sequence $\{t_k\}\to 1$ such that
$\{(\widetilde\theta(t_k)\}$ converges to $\vartheta\in\Delta$ and that 
the known term $F$ is holomorphic at $\vartheta$.
Then the initial solution $w$ admits an analytic continuation 
along  $\gamma$ up to the endpoint $\gamma(1)$.
\end{lemma}
\begin{proof}
Let $(z_k, w_k)$ be the coordinates of $p\circ\widetilde\theta(t_k)$ and
$(z_{\infty }, w_{\infty })$ those of $p(\vartheta )$.
Let $\eta_{\infty } $ be the branch of $p^{-1}$ such that 
$\eta_{\infty }(\vartheta )=\vartheta  $ and let 
$$F\circ\eta(z,w)=
\sum_{r,s=0}^{\infty}
c_{r,s}(z-z_k)^r(w-w_k)^s$$
be the Taylor development of $F\circ\eta$ in a bidisc $\DI(z_{\infty }, w_{\infty}, \rho, \sigma)$
around $(z_{\infty }, w_{\infty })$.

The analytic continuation of $F\circ\eta$ along $\widetilde\theta$ can be concretely carried out
in the underlying $\CI^2$ by a chain of bidiscs and Taylor developments
$\{(\Iu_k, F\circ\eta_k )\}_{k\in\pen}$, where 
\begin{equation}
F\circ\eta_k (z,w)=
\sum_{r,s=0}^{\infty}
c_{r,s,k}(z-z_k)^r(w-w_k)^s,\quad k\in\EN,
\label{taylor}
\end{equation}
$
\gamma
\times
[\omega \circ 
\widetilde\gamma]
(t_k)\in \Iu_{N(k)}$ for every $k$ and some stricly increasing function $N:\EN\to\EN$
(which we call the {\it counting function})
and, for each $k$: $\eta_k$ is a local inverse of the projection mapping $p$ and 
$\eta_{N(k)}:\Iu_{N(k)}\to\Delta $ is the local inverse of $p$
such that $\eta_{N(k)}\circ p(\widetilde\theta (t_k))=\widetilde\theta (t_k)$.

By continuity, $\{c_{r,s,k}\}\to c_{r,s}$ for all $r,s$
as $k\to\infty $, hence we can find $a>0$ and $b>0$
such that the developments in (\ref{taylor}) converge absolutely
and uniformly in the closed bidiscs 
$\overline{\DI(u_k, v_k, a, b)}$.
By Cauchy estimates, this implies that there exists $T\in\ER^+$ such that
$\sum_{r,s=0}^{\infty}
\vert c_{r,s,k}\vert a^r b^s<T$ for all $k\in\EN$; by 
classical complex analysis (see e.g.,$\!\,$ \cite{hille},
theorem 2.5.1) all solutions to the Cauchy problems
\begin{equation}
\begin{cases}
\Omega_k^{\,\prime}(z)
=F\circ\eta_k(z,\Omega_k(z))
\cr
\Omega_k(z_k)=w_k.
\end{cases}
, \quad k\in\EN\cup\{\infty \}
\label{cauchylatest}
\end{equation}
have radii of convergence of at least 
$\sigma :=a(1-e^{-b/(2aT)})$, thus (keeping into account that the counting function 
$N$ is strictly increasing) there exists $\ell\in\EN$
such that $\ell\in N(\EN)$, $v_{\infty }\in \DI(v_\ell,\sigma )$; by continuity,
$\Omega_\ell(z_{\infty })=w_{\infty } $. This means
that $\Omega_\ell$ admits analytic continuation along
$\gamma $ up to $z_{\infty }=\gamma (1)$; since, by the hypothesis of the existence of the analytic
continuation of the initial solution $w$ to the Cauchy problem $\ref{cauchy}$, 
we can in turn construct $\Omega_{\ell}$
by starting from  $w$ and carrying out analytic continuation along $\gamma$, 
we can conclude that  $w$ itself admits analytic continuation along $\gamma $ up to 
$z_{\infty }=\gamma (1)$.
\end{proof}

\vskip0.2cm
Finally, here is our main theorem (notation has been set up and discussed in section \ref{multivalued known terms}):

\begin{theorem}
\label{painleve'} 
Let a Cauchy problem 
for a first order ordinary differential equation 
in the complex domain be given. Suppose the singularities
of the differential equation
to be
contained in a complex-analytic
curve $\mathcal{S}\subset \CI^2$.
Let $\gamma\colon [0,1]\rightarrow\CI$
be an arc starting at the initial point $z_0$ such that the complex line 
$z=\gamma(1)$ is not contained  in $\mathcal{S}$.
Suppose that an analytic continuation 
$\left(R,\pi,j,\omega\right)$ of the initial solution $w$
can be obtained along $\gamma\vert_{[0,1)}$; let
$\widetilde\gamma :[0,1)\to R$ the lifted arc of
$\gamma $ with respect to the natural projection $\pi $: then
there exists (finite or infinite)
$
\lim_{t\to 1} \omega\circ\widetilde\gamma(t)
$.
\end{theorem}
\begin{proof}
Let $\mathcal{B}$, $(\Delta , p)$, $\eta$, $\Sigma$, $(z_0,w_0)$
have the same meaning as discussed in section \ref{multivalued known terms}.
In particular, recall that $\eta $ is a local inverse of $p$, defined  in a 
bidisc $\DI_1\times\DI_2 $ around $(z_0,w_0)$ and with value in the Riemann domain
$\Delta$ resolving the multivaluedness of the known term; hence, viewed from the
underlying $\CI^2$, $F\circ\eta$ is a branch of the multivalued function $F$.
Also, recall that our Cauchy problem is:
$$
\begin{cases}
u^{\prime}(v)
=F\circ\eta(z, w(z))
\cr
w(z_0)=w_0.
\end{cases}
$$

Now
$\mathcal{S}$ is complex-analytic and $\{z=\gamma(1)\}\not\subset \mathcal{S}$,
so 
$$P:=\{w\in\CI : (\gamma(1),w)\in \mathcal{S}\}$$
is discrete in $\CI_w$; hence we can suppose it to be indexed over $\EN$ or a finite subset.

Suppose now, by contradiction, that $\lim_{t\to 1} 
\omega \circ \widetilde\gamma (t)$ does not exist. 
By lemma \ref{fmetric} {\tt A)} or {\tt B)}, according as $P$ is finite or infinite, 
(with $X=\PI^1$, $\alpha =\omega \circ \widetilde\gamma $, 
$\{x_k\}=P\cup\{\infty\}$), there exist:
a sequence $\{t_k\}\to 1$, 
$r>0$, $\varepsilon >0$
and
a finite subset $Q=\{\lambda_{\nu}\}\subset P$,
such that 
$$\{\omega \circ \widetilde\gamma (t_k)\}
\subset 
\overline{\DI(0,r)}
\setminus 
\bigcup_{\lambda_{\nu} \in Q}
\DI(\lambda_\nu, \varepsilon)
.$$
Now, by continuity, there exists $\varrho >0$
such that 
$$
z\in \DI(\gamma (1), \varrho )
\Rightarrow 
\hbox{\rm pr}_{\Cit_w}
\Big[\mathcal{S}\cap \left(\{z\}\times\CI\right)  \Big]
\subset 
\bigcup_{\lambda_{\nu} \in Q}
\DI(\lambda_\nu, \varepsilon);
$$
Set 
$$W:= 
\overline{
\DI(\gamma (1), \varrho )}
\times 
\left[\overline{\DI(0,r)}
\setminus 
\bigcup_{\lambda_{\nu} \in Q}
\DI(\lambda_\nu, \varepsilon)
\right];
$$ 
by construction $W$ is compact in $\CI^2$ and 
$
W\cap\mathcal{S}=\emptyset
$;
also, we may suppose, without
loss of generality, 
$\{\gamma (t_k)\}\subset \DI(\gamma (1), \varrho )$,
implying in turn
\begin{equation}
\{\gamma \times [\omega \circ \widetilde\gamma](t_k) \}
\subset 
W.
\label{dentw}
\end{equation}
Let now $A$ be the holomorphic function on $\CI^2$
such that $\mathcal{S}=A^{-1}(0)$;
the set 
\begin{equation}
\label{discrete}
\mathcal{B}:=
\{\zeta \in R:
A(\pi(\zeta), {\omega }(\zeta )  )=0
\}
\end{equation}
is discrete for otherwise we would have 
$A(\pi(\zeta), {\omega }(\zeta )  )\equiv 0$ for 
all $\zeta \in R$ 
contradicting the hypothesis that $(z_0,w_0)\not\in \mathcal{S}$.
Thus, by continuity. $\widetilde\gamma^{-1}(\mathcal{B})$ is discrete and, by 
(\ref{dentw}),  $\widetilde\gamma^{-1}(\mathcal{B})\cap\{t_k\}$ is finite.
Therefore, by passing to a nearby homotopic arc
if needed, we may suppose 
$\widetilde\gamma([0,1))\cap\mathcal{B}=\emptyset $, implying
\begin{equation}
\gamma
\times
[\omega \circ 
\widetilde\gamma]
([0,1))\cap\mathcal{S}=\emptyset
.
\nonumber
\end{equation}
Hence $F\circ\eta $ admits regular analytic continuation along
$\theta :=
\gamma
\times
[\omega \circ 
\widetilde\gamma]
:[0,1)\to\CI^2$. 
Let $\widetilde\theta :[0,1)\to \Delta  $
be the lifted arc with respect to the projection mapping
$p $. 
By construction and by (\ref{dentw}) we have
$$
\{\widetilde\theta(t_k)\}
\subset 
p^{-1}\left(\{\gamma
\times
\omega \circ 
\widetilde\gamma (t_k))\}\right) 
\subset   
p^{-1}(W).
$$
Since $W$ is compact and $p$ is proper, $p^{-1}(W)$ is compact;
thus, by maybe passing to a subsequence, we may assume
that $\{\widetilde\theta(t_k)\}$
converges to a limit 
$\vartheta  \in p^{-1}(W)\subset \Delta \setminus \Sigma $;  
note that $F$ is holomorphic at $\vartheta $; now a direct application of lemma
\ref{single sequence} allows us to conclude that $w$ admits analytic continuation
up to $\gamma(1)$.
This fact contradicts the hypothesis that 
$
\lim_{t\to 1}\omega\circ\widetilde\gamma(t)
$
does not exist.
\end{proof}
\vskip0.3cm
\noindent
{\bf Remark} Theorem \ref{painleve'} immediately implies that none of the singularities of the solution of the Cauchy problem
(\ref{cauchy}) are essential, for if $\zeta$ were such a singularity, there would exists a path 
$\gamma:[0,1]\to\CI$ connecting $z_0$ to $\zeta$ such that $w$ admits analytic continuation along
$\gamma|_{[0,1)}$ and no continuous extension up to $\gamma(1)$.

\section{Acknowledgement}
The author wishes to thank the referee for several valuable remarks 
and criticisms, which helped to have the quality of the manuscript substantially improved; 
in particular, but with no limitations, the author is grateful for his/her suggestion of including 
counterexample \ref{noproper}.

\end{document}